\documentclass[11pt,a4paper,twoside]{article}

\usepackage{amsfonts}
\usepackage{mathrsfs}
\usepackage{amsthm}
\usepackage{amsmath}%
\usepackage{amssymb}%
\usepackage{fancyhdr}%
\usepackage{ifthen}%
\usepackage{calc}%
\usepackage{lastpage}%
\usepackage{array}%
\usepackage{anysize}%
\usepackage{cases}
\usepackage[pdftex]{hyperref}
\makeatletter
\renewcommand\@biblabel[1]{${#1}.$}
\makeatother

\title{\bf The Radial Part of Brownian Motion with respect to  $\mathcal{L}$-Distance under Ricci Flow}
\author{{\bf Li-Juan Cheng \footnote{\scriptsize School of Mathematical Sciences,  Beijing Normal
University,  Laboratory of Mathematics and Complex
Systems,  Ministry of Education,  Beijing 100875,
The People's Republic of China.  E-mail: chenglj@mail.bnu.edu.cn(L.J. Cheng)
}}}
\date{}

\newtheorem{theorem}{Theorem}[section]

\newtheorem{lemma}[theorem]{Lemma}
\newtheorem{proposition}[theorem]{Proposition}

\newtheorem{definition}[theorem]{Definition}

\newtheorem{remark}[theorem]{Remark}
\newtheorem{example}[theorem]{Example}

\numberwithin{equation}{section} \catcode`@=11

\begin{document}
\maketitle
\begin{abstract}
Let $\{g_t\}_{t\in [0,T)}$ be a family of complete time-depending Riemannian matrics on a manifold which evolves under backwards Ricci flow.
The It\^{o} formula is established for the $\mathcal{L}$-distance of the $g_t$-Brownian motion to a fixed
reference point ($\mathcal{L}$-base).  Furthermore, as an application, we
construct a coupling by parallel displacement which yields  a new proof
of some results of Topping.
\end{abstract}



\def\dint{\displaystyle\int}
\def\ct{\cite}
\def\lb{\label}
\def\ex{Example}
\def\vd{\mathrm{d}}
\def\dis{\displaystyle}
\def\fin{\hfill$\square$}
\def\thm{theorem}
\def\bthm{\begin{theorem}}
\def\ethm{\end{theorem}}
\def\blem{\begin{lemma}}
\def\elem{\end{lemma}}
\def\brem{\begin{remark}}
\def\erem{\end{remark}}
\def\bexm{\begin{example}}
\def\eexm{\end{example}}
\def\bcor{\bg{corollary}}
\def\ecor{\end{corollary}}
\def\r{\right}
\def\l{\left}
\def\var{\text {\rm Var}}
\def\lmd{\lambda}
\def\alp{\alpha}
\def\gm{\gamma}
\def\Gm{\Gamma}
\def\e{\operatorname{e}}
\def\gap{\text{\rm gap}}
\def\dsum{\displaystyle\sum}
\def\dsup{\displaystyle\sup}
\def\dlim{\displaystyle\lim}
\def\dlimsup{\displaystyle\limsup}
\def\dmax{\displaystyle\max}
\def\dmin{\displaystyle\min}
\def\dinf{\displaystyle\inf}
\def\be{\begin{equation}}
\def\de{\end{equation}}
\def\dint{\displaystyle\int}
\def\dfrac{\displaystyle\frac}
\def\zm{\noindent{\bf  Proof.\ }}
\def\endzm{\quad $\Box$}
\def\mO{\mathcal{O}}
\def\mW{\mathcal{W}}
\def\mL{\mathcal{L}}
\def\LC{\mathcal{L}{\rm Cut}}
\def\proclaim#1{\bigskip\noindent{\bf #1}\bgroup\it\  }
\def\endproclaim{\egroup\par\bigskip}
\baselineskip 18pt

\section{Introduction and main result}
\hspace{0.5cm} Let $M$ be a  $d$-dimensional differentiable manifold carrying
 a complete backwards Ricci flow $\{g_{\tau}\}_{\tau\in [0,T)}$, $0<T\leq \infty$,   i.e. a smooth family of Riemannian metrics solving the nonlinear PDE
\begin{align}\label{Ricci-flow}
\frac{\partial g_{\tau}}{\partial \tau }=2{\rm Ric}_{\tau},
\end{align}
 such that $(M,g_{\tau})$ is complete for all $\tau\in [0,T)$, where ${\rm Ric}_{\tau}$ is the Ricci curvature induced by the metric $g_{\tau}$.
According to Perelman \cite{Pe1},
for $0\leq \tau_1<\tau_2<T$,
 Perelman's
$\mL$-length of a differentiable path $\gamma: [\tau_1, \tau_2]\rightarrow M$ is then
defined  by
$$\mL(\gamma):=\int_{\tau_1}^{\tau_2}\sqrt{\tau}\l[|\dot{\gamma}(\tau)|^2_{\tau}+R(\gamma(\tau),\tau)\r]\vd
\tau,$$
 where $R(x,\tau)$ is the scalar curvature at $x\in M$ w.r.t. the metric $
g_{\tau}$.
Define the $\mathcal{L}$-distance between two points $(x,\tau_1)$ and
$(y,\tau_2)$ by
$$Q(x,\tau_1;y,\tau_2)=\inf\l\{\mL(\gamma)|\gamma:[\tau_1,\tau_2]\rightarrow M\ \mbox{ is smooth and}\ \gamma(\tau_1)=x,\gamma(\tau_2)=y \r\}.$$
Note that the $\mathcal{L}$-distance can be negative, and it is in general not a real distance. But it
reduces to the  Riemannian distance in the sense that
$$\lim_{\tau_2\downarrow
\tau_1}2(\sqrt{\tau_2}-\sqrt{\tau_1})Q(x,\tau_1;y,\tau_2)={\rho}^2_{\tau_1}(x,y),$$
 where $\rho_{\tau_1}$ is the Riemmannian distance with respect to
$g_{\tau_1}$.

In this paper, we want to use the comparison theorem to analyze the behavior
of the $g_t$-Brownian motion.
 Let $\nabla^{t}$ and $\Delta_t$ be the Levi-Civita connection and the  Laplace operator associated with the metric $g_t$ respectively.  Let $\mathcal{F}(M)$ (resp. $\mathcal{O}_{t}(M)$) be the
(resp. $g_t$-orthonormal) frame bundle.  Let $\mathbf{p}:
\mathcal{F}(M)\rightarrow M$ be  the canonical projection.  Set $(e_i)_{i=1}^{d}$ be the orthonormal basis on
$\mathbb{R}^d$. For  $t\in [0,T)$ and $u\in \mathcal{O}_t(M)$, let  $H_i(t, u)$ be the $\nabla^{t}$-horizontal lift of $ue_i$ and $(V_{\alpha,\beta}(u))_{\alpha,\beta=1}^{d}$  the canonical vertical
vector fields. Let $(B_t)_{t\geq 0}$ be a standard $\mathbb{R}^d$-valued
Brownian motion on  a complete
filtered probability space $(\Omega,\{\mathscr{F}_t\}_{t\geq 0}, \mathbb{P})$. In this situation, Arnaudon, Coulibaly and
Thalmaier \cite{ACT} constructed the horizontal Brownian motion on $\mathcal{F}(M)$ by solving
the  following Stratonovich SDE
$$\begin{cases}
 \vd u_t=\sqrt{2}\dsum_{i=1}^{d}H_{i}(t, u_t)\circ \vd
B_t^{i}-\frac{1}{2}\dsum_{\alpha,
\beta=1}^{d}\mathcal{G}_{\alpha,\beta}(t,u_t)V_{\alpha \beta}(u_t)\vd t,\medskip\\
u_{s_0}\in \mathcal{O}_{s_0}(M), {\bf p}u_{s_0}=x,
\end{cases}$$
where $\mathcal{G}_{\alpha,\beta}(t,u_t):=\partial_tg_t(u_te_{\alpha},u_te_{\beta}),\  \alpha,\beta=1,2,\cdots, d.$
They have shown that the drift term in the equality is essential to ensure  $u_t\in
\mathcal{O}_{t}(M)$ for all $t\in [0,T)$.  Moreover, this process is non-explosive up to $T$ since $g_t$ is the complete backward Ricci flow  (see \cite[Theorem 1]{Kuwada}).
The $g_t$-Brownian motion
 is  then defined by $X_t={\bf p} u_t$.
For a given reference point ($\mathcal{L}$-base) $(o,0)$, $o\in M$,
define the radius function
$$Q(x,t):=Q(o,0; x,t),\ \ \mbox{for}\ \ x\in M $$
as the $\mathcal{L}$-distance between $(x,t)$ and $(o,0)$.
If $Q$ is smooth, then by the It\^{o} formula, we have
$$Q(X_t,t)=Q(x,s_0)+\sqrt{2}\int_{s_0}^t\l<\nabla^{t}_1Q(X_s,s),u_s\vd B_s\r>_{s}+\int_{s_0}^t\l[\Delta_{s}Q+{\partial_ s}Q\r](X_s,s)\vd s,\ \ t\geq s_0,$$
where  $\l<\cdot,\cdot\r>_s:=g_s(\cdot,\cdot)$. However, in general, $Q$ is not smooth on whole manifold, so that it is even not clear whether $Q(X_t,t)$ is a semimartingale.  The purpose of this paper is to prove that $Q(X_t,t)$
is indeed a semimartingale and establish the It\^{o} formula for it.

We would like to indicate that when the metric is independent of $t$,  the semimartingale property for the radial part of the Brownian motion w.r.t. Riemannian distance was first proved  by Kendall  \cite{Kendall}, which  is fundamental to analyze the Brownian motion on a Riemannian manifold.  Especially, Kendall's It\^{o} formula was  applied to the construction of  coupling processes on manifolds (see \cite[Chapter 2]{Wbook1}). For the time-inhomogeneous case, Kuwada and Philipowski \cite{Kuwada} shows that  the radial part of Brownian motion $\rho_{t}(o,X_t)$,   the Riemannian distance from $o$ to $X_t$ w.r.t. $g_t$, is a semimartingale, which is applied to the non-explosion of the $g_t$-Brownian motion.
See \cite{Kuwada2} for more discussions in this direction.

  By using an approximation approach to the $\mathcal{L}$-cut-locus, we are able  to extend Kendall's It\^{o} formula to the $g_t$-Brownian motion as follows.
 \bthm\lb{Ito-formula}
Let $X_t$ be a $g_t$-Brownian motion starting at  time $s_0\in (0,T)$. Then there exists a
non-decreasing continuous process $L$ which increases only when
$(X_t, t)\in \mL{\rm Cut}((o,0))$ such that
\be\lb{Ito}
\vd
Q(X_t,t)=\sqrt{2}\l<\nabla^{t}_1Q(X_t, t), \vd B_t\r>_{t}+\l[\Delta_{t}Q+{\partial_ t}Q\r](X_t,t)\vd
t-\vd L_t,\ t\in (s_0,T),
\de where $\nabla ^{t}_1Q(\cdot,t)$ and
$\Delta_{t}Q(\cdot,t)$ are defined to be zero where $Q(\cdot,t)$
fails to be differentiable. In particular, $Q(X_t,t)$ is
semimartingale. \ethm

As an application,  we will construct a coupling  of  $g_{\bar{\tau}_1t}$- and $g_{\bar{\tau}_2t}$-Brownian motions by parallel displacement, where $0<\bar{\tau}_1\leq \bar{\tau}_2<T$ and $0<s\leq t<T/\bar{\tau}_2$.  It is well-known that the coupling method is a useful tool both in stochastic analysis and
geometric analysis.
   We will use this tool to obtain the
martingale property of $Q(X_{\bar{\tau}_1t},\bar{\tau}_1t;\tilde{X}_{\bar{\tau}_2t},\bar{\tau}_2t)$.
 We would like to point out that
very recently,
Kuwada and Philipowski \cite{Kuwada2} constructed a coupling via approximation by geodesic random work, and  applied it to proving the monotonicity of the normalized
$\mL$-transportation cost between solutions of the heat equation.
 Here, we present an alternative construction such that a large number of estimates presented in \cite{Kuwada2} are avoided. When $g_t$ is independent of $t$, our construction is due to to Wang \cite{Wbook1}.

The rest parts of the paper is organized as follows. In Section 2, we introduce the $\mathcal{L}$-cut-locus, and some properties of it.
In  Section 3, we prove  Theorem \ref{Ito-formula}. In the final section,
we construct a coupling of $g_{\bar{\tau}_1t}$- and $g_{\bar{\tau}_2t}$-Brownian motions by parallel
displacement, which leads to a new proof of  the normalized $\mathcal{L}$-transportation cost inequality introduced by Topping \cite{Topping}.

For readers' convenience,   we will take the same notations as in \cite{Topping}.

\section{Definition and properties of  $\mathcal{L}$-cut-locus}
\hspace{0.5cm}
Recall that $\{g_{\tau}\}_{\tau\in [0,T)}$ is  a complete backwards Ricci flow.
Let
$$\Upsilon=\{(x,\tau_1;y,\tau_2)\ |\ x,y \in M \ \mbox{and}\ 0\leq\tau_1<\tau_2<T \}.$$
Similar to the
Riemannian distance, in general, $Q$  fails to be smooth on some subset
$\mL{\rm Cut}$  defined as follows.
For $\tau_1, \tau_2\in [0,T)$ with $\tau_1<\tau_2$, $x\in M$ and $Z\in T_xM$,
we define the $\mL$-exponential map
$\mL_{\tau_1,\tau_2}\exp_x:T_xM\rightarrow M$ by
$\mL_{\tau_1,\tau_2}\exp_x(Z)=\gamma(\tau_2)$, where $\gamma$ is a
unique $\mL$-geodesic (see Remark \ref{rem1}) starting from $x$ at time $\tau_1$ with the initial
condition
$\lim_{\tau\downarrow\tau_1}\sqrt{\tau}\dot{\gamma}(\tau)=Z$. Note
that the $\mathcal{L}$-geodesic also induces a notation of $\mathcal{L}$-Jacobi fields (see e.g. \cite[Chapter 7]{Ricci}). It is convenient to define
\begin{align*}
\Omega(x,\tau_1;\tau_2)=\bigg\{Z\in T_xM\ \bigg|\ \
\begin{split}&\mbox{$\gamma:[\tau_1,\tau_2]\rightarrow M$ \ defined
by
$\gamma(\tau)=\mL_{\tau_1,\tau}\exp_x(Z)$}\\
&\mbox{ is a unique minimising $\mL$-geodesic}
 \end{split}\ \ \bigg\}.\end{align*}
Let $$\Omega^*(x,\tau_1;T):=\bigcap_{\tau_2\in (\tau_1, T)}\Omega(x,\tau_1;\tau_2).$$ For any $Z\in
T_xM\setminus\Omega^*(x,\tau_1;T)$, let $\bar{\tau}(x,\tau_1;
Z)=\sup\{\tau\in (\tau_1,T)\ |\ Z\in \Omega(x, \tau_1;\tau)\}$. Then,
 the $\mathcal{L}$-cut-locus is defined as follows:
 \begin{displaymath} \mL {\rm
Cut}= \Biggr\{(x,\tau_1; y, \tau')\ \ \Biggr| \
\begin{array}{ll}
x\in M,\tau_1\in [0, T);&\\
 y=\mL_{\tau_1,\tau'}\exp_x(Z)\ \mbox{for some}\
Z\in T_xM\setminus\Omega^*(x,\tau_1;T);& \\
\tau'= \bar{\tau}(x,\tau_1;Z)\in [\tau_1,T)&
\end{array}\hspace{-0.4cm} \Biggr\}.
\end{displaymath}
Let
$$\LC((x,\tau_1))=\{(y,\tau_2)\in M\times (\tau_1,T)\ |\ (x,\tau_1;y,\tau_2)\in \LC\}.$$
The set $\mL{\rm Cut}$ can be decomposed into
two parts: the first consists of points $(x,\tau_1;y,\tau_2)$ such
that there exists more than one minimizing $\mL$-geodesic
$\gamma:[\tau_1,\tau_2]\rightarrow M$ with $\gamma(\tau_1)=x$ and
$\gamma(\tau_2)=y$, and the second is the set of points $(x,\tau_1;
y,\tau_2)$ such that $y$ is conjugate to $x$ (with respect to
$\mL$-Jacobi fields) along a minimizing $\mL$-geodesic $\gamma:
[\tau_1,\tau_2]\rightarrow M$ with
$\gamma(\tau_1)=x,\gamma(\tau_2)=y$.

 The following   important properties about the $\mL$-cut-locus and $Q$  can be found in  \cite[Lemma\ 7.27]{Ricci} and \cite[lemma 2.14]{Ye}.
\begin{proposition}
\lb{Lcut}
\begin{enumerate}
  \item [$(1)$] The two sets $\LC$ and  $\LC((o,0))$ are  closed
 of measure zero in   $\Upsilon$ and $M\times [0,T)$ respectively. Moreover, for any $t\in [0,T)$, the set
 $$\LC_t(o):=\{x\in M: (x,t)\in \LC((o,0))\}$$
 is of measure zero in $M$.
  \item [$(2)$] The function $Q$ is smooth on $\Upsilon\backslash \LC$.
  \item [$(3)$] 
  If we associate to each point $(x,\tau_1;y,\tau_2)\in \Upsilon\backslash
  \LC$ the vector $Z\in \Omega(x,\tau_1;\tau_2)\in T_xM$ for which
  $\mathcal{L}_{\tau_1,\tau_2}\exp_x(Z)=y$, then $Z$ depends smoothly on
  $(x,\tau_1, y,\tau_2)$.
  \item [$(4)$] On $\Upsilon\backslash \LC$, we have
   \begin{align*}
   &\frac{\partial Q}{\partial \tau_1}(x,\tau_1;y, \tau_2)=\sqrt{\tau_1}\l(|\dot{\gamma}(\tau_1)|_{\tau_1}^2-R(x,\tau_1)\r); \nabla^{\tau_1}_1Q(x,\tau_1;
   y,\tau_2)=-2\sqrt{\tau_1}\dot{\gamma}(\tau_1);\\
   &\frac{\partial Q}{\partial \tau_2}(x,\tau_1;y, \tau_2)=\sqrt{\tau_2}\l(|\dot{\gamma}(\tau_2)|_{\tau_2}^2-R(x,\tau_2)\r); \nabla^{\tau_2}_2Q(x,\tau_1;
   y,\tau_2)=2\sqrt{\tau_2}\dot{\gamma}(\tau_2),
  \end{align*}
   where  $\gamma: [\tau_1,\tau_2]\rightarrow
   M$ is the minimizing $\mathcal{L}$-geodesic from $x$ to $y$ and $ \nabla^{\tau_1}_1$ $($resp. $\nabla^{\tau_2}_2$$)$  denotes the gradient with respect to the variable $x$ $($resp. the variable $y$$)$ by using the metric $g_{\tau_1}$ $($resp. $g_{\tau_2}$$)$.
\end{enumerate}
\end{proposition}
 Since $(X_t)$ is generated by a non-degenerated operator, the following  is a direct consequence of Proposition \ref{Lcut}(1).
 \blem\lb{pmea0} Suppose $X_t$ is a $g_t$-Brownian motion starting at time $s_0\in (0,T)$. The set $\{t\in
[s_0,T)\ |\ (X_t,t)\in \LC((o,0))\}$ has Lebesgue measure zero almost
surely. \elem
\begin{proof}
According to \cite[Lemma 2]{Kuwada},  for  any starting point $x\in M$, the law of $X_t$ under
$\mathbb{P}^x$ is absolutely continuous with respect to the $g_t$-Riemannian volume
measure. Moreover, by Proposition \ref{Lcut}(1), $\LC_t(o)$ measures zero, we have
\begin{align*}
    \mathbb{E}^{x}\l[\int_{s_0}^T1_{\{(X_t,t)\in \LC((o,0))\}}\vd t\r]=\int_{s_0}^T\mathbb{P}^x((X_t,t)\in \LC(o,0))\vd t=\int_{s_0}^T\mathbb{P}^x(X_t\in \LC_t(o))\vd t=0,
\end{align*}
it follows that $\int_{s_0}^T1_{\{(X_t,t)\in \LC(o,0)\}}\vd t=0, $ a.s..

\end{proof}


\section{Proof of Theorem \ref{Ito-formula}}
\hspace{0.5cm}
Since $X_t$ is non-explosive  before the life time of the metric family,  by a localization argument, it is sufficient to consider the case of compact $M$. Thus, in this section, we  assume that $M$ is compact and
 $[0,T]$ is a finite interval.
 We first state the It\^{o}
formula for smooth functions.
\begin{lemma}\label{obsence-L-cut-locus} Suppose $X_t$ is a $g_t$-Brownian motion starting at time $s_0\in (0,T)$.
Let $f$ be a smooth function on $M\times [s_0,T]$. Then,
\begin{align*}
\vd f(X_t,t)=\frac{\partial f}{\partial t}(X_t,t)\vd t+\Delta_{t}f(X_t,t)\vd t+\sqrt{2}\sum_{i=1}^{d}u_te_if(X_t,t)\vd B_t^i,\ \ s_0<t\leq T.
\end{align*}
\end{lemma}
According to Proposition \ref{Lcut}(2) and Lemma \ref{obsence-L-cut-locus}, we see that, if $(X_t,t)$ stays away from  the $\mL$-cut-locus of $(o,0)$, then
\begin{align}\label{normalIto}
\vd Q(X_t,t)=\vd
\beta_t+\l[\Delta_{t}Q+\frac{\partial Q}{\partial t}\r](X_t,t)\vd
t,\ t\in [s_0,T],
\end{align}
where $\beta_t$ is the martingale term given by
$$\vd\beta_t:=\sqrt{2}\sum_{i=1}^d (u_se_i)Q(X_t,t)\vd B_t^i.$$
 By
Proposition \ref{Lcut}(4), the quadratic variation of the martingale $\beta_t$
is computed  as follows
\begin{align*}
\vd\l<\beta\r>_t=2\sum_{i=1}^{d}[(u_te_i)Q(X_t,t)]^2\vd
t=2|\nabla_1^{t}Q(X_t,t)|_{t}^2\vd t
=8t|\dot{\gamma}(t)|_{t}^2\vd t,
\end{align*}
where $\gamma: [0,t]\rightarrow M$ is the minimal $\mL$-geodesic
from $o$ to $X_t$ and $|\cdot|_t:=\sqrt{\l<\cdot,\cdot\r>_t}$. Thus, $\beta_t$ can be represented by
\begin{equation}\label{martingale}
   2\sqrt{2t}|\dot{\gamma}(t)|_{t} \vd b_t,
\end{equation}
 where $b_t$ is a
standard one-dimensional Brownian motion. We will explain in Remark \ref{rem1} that the coefficient $2\sqrt{2t}|\dot{\gamma}(t)|_{t}$ is not a constant, which is  different from the fixed metric case.

    Next, to control the drift term of \eqref{normalIto}, we need the comparison theorem, which is a combination of the following two lemmas (see \cite[Lemma\ 7.45]{Ricci} and \cite[Lemma\ 7.13]{Ricci}).
\begin{lemma}[\cite{Ricci}]\label{comparison-lemma}
For $0\leq \tau_1<\tau_2\leq T$, let $\gamma:[\tau_1,\tau_2]\rightarrow M$ be a minimal $\mathcal{L}$-geodesic from $p$
to $q$. At $(q,\tau_2)$ the $\mathcal{L}$-distance satisfies
\begin{equation*}
    \frac{\partial}{\partial \tau_2}Q(p,\tau_1; q,\tau_2)+\Delta_{\tau_2}Q(p,\tau_1;\cdot,\tau_2)(q)\leq
    \frac{d}{\sqrt{\tau_2}-\sqrt{\tau_1}}-\frac{1}{2(\tau_2-\tau_1)}Q(p,\tau_1;q,\tau_2).
\end{equation*}
\end{lemma}
Since $M$ is compact,  there exists some constant $C_0<\infty$ such that
\begin{equation}\label{ineq1}
    \max_{(x,\tau)\in M\times [0,T]}|{\rm Rm}_{\tau}|(x)\vee|{\rm Ric}_{\tau}|(x)\leq C_0.
\end{equation}
We obtain the  lower bound for $Q$ from \eqref{ineq1}.
\begin{lemma}[\cite{Ricci}]\label{estimation-of-Q}
For $0\leq \tau_1<\tau_2\leq T$, let $\gamma:[\tau_1,\tau_2]\rightarrow M$ be a minimal $\mathcal{L}$-geodesic. Then,
\begin{align*}
Q(\gamma(\tau_1),\tau_1; \gamma(\tau_2),\tau_2) \geq \frac{e^{-2C_0(\tau_2-\tau_1)}}{2(\sqrt{\tau_2}-\sqrt{\tau_1})}\rho_{\tau_1}(\gamma(\tau_1),\gamma(\tau_2))^2-&\frac{2}{3}dC_0(\tau_2^{3/2}-\tau_1^{3/2}). 
\end{align*}

\end{lemma}

By Lemmas \ref{comparison-lemma} and \ref{estimation-of-Q}, we have, for $(x,t)\notin \LC((o,0))$,
\begin{align*}
\frac{\partial}{\partial t}Q(x,t)+\Delta_{t}Q(x,t)&\leq \frac{d}{\sqrt{t}}-\frac{1}{2t}Q(x,t)
\leq \frac{d}{\sqrt{t}}+\frac{1}{2t}Q(x,t)^-\leq \frac{d}{\sqrt{t}}+\frac{dC_0}{3t}t^{3/2}.
\end{align*}
Define $V(t):=\frac{4d}{\sqrt{t}}+\frac{dC_0}{2}\sqrt{t}$. It is easy to see  that
\begin{equation}\label{compa}
   \frac{\partial}{\partial t}Q(x,t)+\Delta_{t}Q(x,t)\leq V(t).
\end{equation}

Now, we turn to construct a closed set such that it is disjoint with $\LC$. To this end,  we  consider the product manifold $M\times \mathcal [0,T]$ equipped  with  metric $\vec{g}$:
for $x\in M$ and $t\in [0,T]$,
\begin{align*}
\vec{g}(X,Y)(x,t):=g_{t}(X,Y); \ \vec{g}\l(X,\frac{\vd}{\vd t}\r):=0;\ \vec{g}\l(\frac{\vd}{\vd t},\frac{\vd}{\vd t}\r):=1,\ \
X,Y\in T_xM.
\end{align*}
 Given a path $\gamma: [\tau_1,\tau_2]\rightarrow
M$ with $[\tau_1,\tau_2]\subset [0,T]$ and $\gamma(\tau_1)=x,\
\gamma(\tau_2)=y$. The length of the graph $\tilde{\gamma}:
[\tau_1,\tau_2]\rightarrow M\times [0,T]$, defined by
$\tilde{\gamma}(\tau):=(\gamma(\tau),\tau)$, is given by
$$L_{\vec{g}}(\tilde{\gamma})=\int_{\tau_1}^{\tau_2}\l|\frac{\vd \tilde{\gamma}}{\vd \tau}\r|_{\vec{g}}\vd \tau=\int_{\tau_1}^{\tau_2}\sqrt{\l|\frac{\vd \gamma}{\vd \tau}\r|_{\tau}^2+1}\ \vd \tau,$$
where $|\cdot|_{\vec{g}}$ is the norm w.r.t. the metric $\vec{g}$.
Then the distance between
$(x,\tau_1)$ and $(y,\tau_2)$ can be defined as before, namely,
$${\bf d}_{\vec{g}}(x,\tau_1 ; y,\tau_2):=\inf_{\tilde{\gamma}}L_{\vec{g}}(\tilde{\gamma}).$$
 Let us
define a set $A$ by
 \begin{displaymath}
 A= \Biggr\{(y,\tau_1; z, \tau_2)\in \Upsilon\ \ \Biggr| \
\begin{array}{ll}
y,z\in M, \tau_2\in [s_0, T);&\\
 \tau_1=\tau_2/2;& \\
Q(y,\tau_1;z,\tau_2)+Q(y,\tau_1)=Q(z,\tau_2)&
\end{array}\hspace{-0.4cm} \Biggr\}.
\end{displaymath}
Note that $A$ is closed and hence compact since  $Q(y,\tau_1 ; z,\tau_2)$
is continuous in $(y,\tau_1,z,\tau_2)$. Moreover, for any $(y,\tau_1; z,\tau_2)\in A$, the point
$(y,\tau_1)$ is on a minimal $\mathcal{L}$-geodesic joining $(o,0)$ and $(z,\tau_2)$. In particular, the
symmetry of the $\mathcal{L}$-cut-locus implies that
$$A\cap\LC=\varnothing.$$
Combining this with Proposition \ref{Lcut} (1), we obtain
$$\delta_1:={\bf d}_{\vec{g}}\otimes{\bf d}_{\vec{g}}(A,\LC)>0,$$
where ${\bf d}_{\vec{g}}\otimes {\bf d}_{\vec{g}}$ is a metric on ${\Upsilon}^2$.

The following lemma is essential to the proof of Theorem \ref{Ito-formula}.
\begin{lemma}\label{semimartingale-V}
Let $(x_0,t_0)\in \LC ((o,0))$ and $\delta \in (0,\delta_1)$. Let $(X_t)$ be the $g_t$-Brownian motion starting
from $x_0$ at time $t_0$. Let $\tilde{T}=T\wedge \inf\{t\geq t_0| {\bf d}_{\vec{g}}(x_0,t_0; X_t,t)=\delta\}.$ Then
\begin{align*}
\mathbb{E}\l[Q(X_{t\wedge \tilde{T}},t\wedge \tilde{T})-Q(x_0,t_0)-\int_{t_0}^{\tilde{T}\wedge t}V(s)\vd s\r]\leq 0.
\end{align*}
\end{lemma}
\begin{proof}
We construct a point $\tilde{o}$ as follows: we choose a minimizing $\mathcal{L}$-geodesic
$\gamma$ from $(o,0)$ to $(x_0,t_0)$ and choose $(\tilde{o},\tilde{t})$ on $\tilde{\gamma}$
such that $\tilde{t}=t_0/2$.
Then by the construction of $A$, we have $(\tilde{o},\tilde{t};x_0,t_0)\in A$. Moreover,
for all $t\in [t_0,\tilde{T}]$, we have
$${\bf d}_{\vec{g}}\otimes {\bf d}_{\vec{g}}((\tilde{o},\tilde{t}; x_0,t_0),(\tilde{o},\tilde{t};X_t,t))={\bf d}_{\vec{g}}((x_0,t_0),(X_t,t))<\delta_1,$$
which implies $(X_t,t)\notin \LC((\tilde{o},\tilde{t}))$. Let
$$Q^+(x,t)=Q(o,0;\tilde{o},\tilde{t})+Q(\tilde{o},\tilde{t};x,t).$$
Since $\tilde{o}$ lies on a minimizing $\mathcal{L}$-geodesic from $(o,0)$ to $(x_0,t_0)$,
we have $Q^+(x_0,t_0)=Q(x_0,t_0)$.
Moreover, by the triangle inequality,
$$Q^+(x,t)\geq Q(x,t),\ \ \mbox{for all}\ (x,t)\in M\times [0,T].$$
According to Lemmas \ref{comparison-lemma} and  \ref{estimation-of-Q},
$$\l(\Delta_tQ^++\frac{\partial Q^+}{\partial t}\r)(x,t)\leq \frac{d}{\sqrt{t}-\sqrt{t_0/2}}+\frac{dC_0}{3({t}-{t_0/2})}[t^{3/2}-(t_0/2)^{3/2}]\leq \frac{4d}{\sqrt{t}}+\frac{dC_0}{2}\sqrt{t}=V(t)$$
holds if $(\tilde{o},\tilde{t};x,t)\notin \LC$. Then,
\begin{align*}
&Q(X_{t\wedge \tilde{T}},t\wedge \tilde{T})-Q(X_{t_0},t_0)-\int_{t_0}^{t\wedge \tilde{T}}V(s)\vd s\\
&\leq Q^+(X_{t\wedge \tilde{T}},t\wedge \tilde{T})-Q^+(X_{t_0},t_0)-\int_{t_0}^{t\wedge \tilde{T}}V(s)\vd s\\
&\leq  Q^+(X_{t\wedge \tilde{T}},t\wedge \tilde{T})-Q^+(X_{t_0},t_0)-\int_{t_0}^{t\wedge \tilde{T}}\l(\Delta_{s}Q^++\frac{\partial}{\partial s}Q^+\r)(X_s,s)\vd s.
\end{align*}
Since $Q^+$ is smooth at $(X_t,t)$ for $t\in [t_0,\tilde{T}]$, the last term is a martingale. Hence the claim follows.
\end{proof}

For $\delta\in (0,\delta_1)$, we define a sequence of stopping times $(S_n^{\delta})_{n\in \mathbb{N}}$
and $(T^{\delta}_n)_{n\in \mathbb{N}}$
\begin{align*}
    &T_0^{\delta}:=s_0;\\
    &S_n^{\delta}:=\{t\geq T_{n-1}^{\delta}| (X_t,t)\in \LC((o,0))\};\\
    &T_n^{\delta}:=T\wedge \inf\{t\geq S_n^{\delta}\ |\ {\bf d}_{\vec{g}}(X_t,t; X_{S_n^{\delta}},S_n^{\delta})=\delta\}.
\end{align*}
Note that these are well-defined because $\LC$ and
$$\{(x,t)\ |\ {\bf d}_{\vec{g}}(x,t; y,s)=\delta,\ s_0\leq s\leq t\leq T, \ y\in M\}$$
 are closed.
\begin{lemma}\label{prop1}
The process $Q(X_t,t)-\int_{s_0}^{t}V(s)\vd s$ is a supermartingale.
\end{lemma}
\begin{proof}
Due to the strong Markov property of the $g_t$-Brownian motion, it suffices to show that
for all deterministic starting point $(x_0,t_0)\in M\times [s_0,T]$ and all $t\in [s_0,T]$,
$$\mathbb{E}\l[Q(X_t,t)-Q(X_{t_0},t_0)-\int_{t_0}^tV(s)\vd s\r]\leq 0.$$
We first observe from Lemma \ref{semimartingale-V} and \eqref{normalIto} that for all $n\in \mathbb{N}$,
\begin{equation*}
    \mathbb{E}\l[Q(X_{t\wedge S_n^{\delta}},t\wedge S_n^{\delta})-Q(X_{t\wedge T_{n-1}^{\delta}},t\wedge T_{n-1}^{\delta})-\int_{t\wedge T_{n-1}^{\delta}}^{t\wedge S_{n}^{\delta}}V(s)\vd s\bigg |\mathscr{F}_{T_{n-1}^{\delta}}\r]\leq 0,
\end{equation*}
and
\begin{equation*}
    \mathbb{E}\l[Q(X_{t\wedge T_n^{\delta}},t\wedge T_n^{\delta})-Q(X_{t\wedge S_n^{\delta}},t\wedge S_n^{\delta})
    -\int_{t\wedge S_n^{\delta}}^{t\wedge T_n^{\delta}}V(s)\vd s\bigg | \mathscr{F}_{S_n^{\delta}}\r]\leq 0.
\end{equation*}
It remains to show that $T_n\rightarrow T$ as $n\rightarrow\infty$. If
$$\lim_{n\rightarrow \infty}T_n=:T_{\infty}<T,$$
then $T_n^{\delta}-S_n^{\delta}$ converges to  0 as $n\rightarrow \infty$. In addition,
${\bf d}_{\vec{g}}(X_{S_n^{\delta}},S_n^{\delta}; X_{T_n^{\delta}},T_n^{\delta})=\delta$ must hold for
infinitely many $n\in \mathbb{N}$. It contradicts to the fact that $X_t$ is uniformly continuous on $[0,T]$.
\end{proof}

\begin{lemma}\label{lem1}
$\lim_{\delta\rightarrow 0}\sum_{n=1}^{\infty}|T_n^{\delta}-S_n^{\delta}|=0$ almost surely.
\end{lemma}
\begin{proof}
For $\delta>0$, define
\begin{align*}
    & E_{\delta}=\{t\in [s_0,T]\ |\  \mbox{ there\  exist}\  t'\in [s_0,T] \ \mbox{satisfying} \ |t-t'|\leq \delta\ \mbox{ and}\   (X_{t'},t')\in \LC ((o,0))\},\\
    & E=\{t\in [s_0,T]\ |\ (X_t,t)\in \LC((o,0))\}.
\end{align*}
Since the map $t\rightarrow (X_t,t)$ is continuous and $\LC$ is closed, the set $E$ is closed and hence $E=\bigcap_{\delta>0}E_{\delta}$ holds. By the definitions of $S_n^{\delta}$ and $T_n^{\delta}$,
$$E\subset \bigcup_{n=1}^{\infty}[S_n^{\delta},T_n^{\delta}]\subset E_{\delta},$$
which, together with the monotone convergence theorem, implies
$$\lim_{\delta\rightarrow 0}\sum_{n=1}^{\infty}|T_n^{\delta}-S_n^{\delta}|\leq \lim_{\delta\rightarrow 0}\int_{s_0}^T1_{E_{\delta}}(t)\vd t=\int_{s_0}^T1_E(t)\vd t=0, \ \mbox{a.e.}, $$
 where the last equality comes from Proposition \ref{Lcut}(1).

\end{proof}

\begin{lemma}\label{lem2}
The martingale part of $Q(X_t,t)$ is
$$\l<\nabla_1^{t}Q(X_t, t),u_t\vd B_t\r>_{t}=\sum_{i=1}^du_te_iQ(X_t,t)\vd B_t^i.$$
\end{lemma}
\begin{proof}
By the martingale representation theorem, there exists an $\mathbb{R}^d$-valued process $\eta$
such that the martingale part of $Q(X_t,t)$ equals to $\int_0^t\eta_s\vd B_s$. Let
$$N_t:= \int_{s_0}^t \eta_s\vd B_s-\sum_{i=1}^d(u_te_i)Q(X_t,t)\vd B_t^i.$$
Using the stopping times $S_n^{\delta}$ and $T_n^{\delta}$, the quadratic variation  $\l<N\r>_T$
of $N$ is expressed as follows:
$$\l<N\r>_T=\sum_{i=1}^d\sum_{n=1}^{\infty}\l(\int_{T_{n-1}^{\delta}\wedge T}^{S_n^{\delta}\wedge T}|\eta_t^i-(u_te_i)Q(X_t,t)|^2\vd t+\int_{S_n^{\delta}\wedge T}^{T_n^{\delta}\wedge T}|\eta_t^i-(u_te_i)Q(X_t,t)|^2\vd t\r).$$
Since $(X_t,t)\notin \LC ((o,0))$ if $ t\in (T_{n-1}^{\delta},S_n^{\delta})$, the It\^{o} formula yields
$$\int_{T_{n-1}^{\delta}\wedge T}^{S_n^{\delta}\wedge T}|\eta^i_t-(u_te_i)Q(X_t,t)|^2\vd t=0$$
for $n\in \mathbb{N}$ and $i=1,2,\cdots,d$. For the second term on the right, since
the manifold is compact, there exits a constant $C>0$, such that
\begin{align*}
 \sum_{i=1}^d\sum_{n=1}^{\infty}\int_{S_{n}^{\delta}\wedge T}^{T_{n}^{\delta}\wedge T}|\eta_t^i-u_te_iQ(X_t,t)|^2\vd t &\leq \int_{\bigcup_{n=1}^{\infty}[S_n^{\delta},T_n^{\delta}]}(|\eta_t|^2+4|t\dot{\gamma}(t)|_t^2)\vd t\\
 &\leq \int_{\bigcup_{n=1}^{\infty}[S_n^{\delta},T_n^{\delta}]}(|\eta_t|^2+C)\vd t.
\end{align*}
Since $\eta_t$ is locally square-integrable on $[s_0,T]$ almost surely, we obtain $\l<N\r>_{T}=0$ by  Lemma  \ref{lem1}, which yields
the conclusion.
\end{proof}

\begin{proof}[{ Proof of Theorem \ref{Ito-formula}}]
 Now, we can  conclude the proof of Theorem \ref{Ito-formula}. Set $I_{\delta}:=\bigcup_{n=1}^{\infty}[S_n^{\delta},T_n^{\delta}]$. Let
\begin{align*}
    L_t^{\delta}:=&-Q(X_t,t)+Q(X_{s_0},s_0)+\sum_{i=1}^{d}\int_{s_0}^t(u_se_i)Q(X_s,s)\vd B_s^i\\
    &+\int_{[s_0,t]\setminus I_{\delta}}\l[\Delta_{s}Q+\frac{\partial Q}{\partial s}\r](X_s,s) \vd s+\int_{[s_0,t]\cap I_{\delta}}V(s)\vd s.
\end{align*}
By \eqref{normalIto},  $L_t^{\delta}$ is a increasing process which increases only when $t\in I_{\delta}$. Moreover, we have
\begin{align}\label{add-1}
    &Q(X_t,t)-Q(X_{s_0},s_0) -\sum_{i=1}^d \int_{s_0}^t(u_se_i)Q(X_s,s)\vd B_s^i-\int_{s_0}^{t}\l[\Delta_{s}Q+\frac{\partial Q}{\partial s}\r](X_s,s)\vd s+L_t^{\delta}\nonumber\\
    =&-\int_{[s_0,t]\setminus I_{\delta}}\l[\Delta_{s}Q+\frac{\partial Q}{\partial s}\r](X_s,s)\vd s-\int_{[s_0,t]\cap I_{\delta}}V(s)\vd s.
\end{align}
From \eqref{compa}, we  obtain
\begin{equation*}
    \l|\int_{[s_0,t]\setminus  I_{\delta}}\l[\Delta_{s}Q+\frac{\partial Q}{\partial s}\r](X_s,s)\vd s+\int_{[s_0,t]\cap I_{\delta}}V(s)\vd s\r|\leq 2\int_{[s_0,t]\cap I_{\delta}}V(s)\vd s,
\end{equation*}
and $V(s)$ is bounded on $[s_0,t]\cap I_{\delta}$. Then by Lemma \ref{lem1}, the right hand of \eqref{add-1} converges to $0$
as $\delta \rightarrow 0$. Thus, $L_t:=\lim_{\delta \downarrow 0}L_t^{\delta}$ exists for all $t\in [s_0,T]$
almost surely and hence \eqref{Ito} holds. Finally,  it is easy to see that $L_t$  increases only when $(X_t,t)\in \LC ((o,0))$ from the corresponding property of $L_t^{\delta}$.
\end{proof}

\section{ Coupling for $g_{\bar{\tau}_1t}$- and $g_{\bar{\tau}_2t}$- Brownian motions}
\hspace{0.5cm}
 First, we introduce  some basic notations concerning the space-time parallel displacement.
\begin{definition}[space-time parallel vector field]\label{def1}
 For $0<\tau_1<\tau_2< T$, let  $\gamma:[\tau_1,\tau_2]\rightarrow M$
be a smooth curve. We say that a vector field $Z$ along $\gamma$ is space-time parallel if
\begin{equation}\label{eq4}
    \nabla^{\tau}_{\dot{\gamma}(\gamma)}Z(\tau)=-{\rm Ric}^\sharp_{\tau}(Z(\tau))
\end{equation}
holds for all $\tau\in [\tau_1,\tau_2]$, where
 ${\rm Ric}^\sharp_{\tau}$ is defined by regarding the $g_{\tau}$-Ricci curvature as a $(1,1)$-tensor.
\end{definition}
Since \eqref{eq4} is a linear first order ODE, for any $\xi\in T_{\gamma(\tau_1)}M$, there exists a unique space-time
parallel vector field $Z$ along $\gamma$ with $Z(\tau_1)=\xi$.
Note that
whenever $Z$ and $Z'$ are space-time parallel vector fields along a curve $\gamma$, their $g_{\tau}$-inner product
is constant in $\tau$:
\begin{align}\label{isometry}
    \frac{\vd}{\vd \tau}\l<Z(\tau),Z'(\tau)\r>_{\tau}&=\frac{\partial}{\partial \tau}g_{\tau}(Z(\tau),Z'(\tau))+\l<\nabla^{\tau}_{\dot{\gamma}(\tau)}Z(\tau),Z'(\tau)\r>_{\tau}+\l<Z(\tau),\nabla^{\tau}_{\dot{\gamma}(\tau)}Z'(\tau)\r>_{\tau}\nonumber\\
    &=2{\rm Ric}_{\tau}(Z(\tau),Z'(\tau))-{\rm Ric}_{\tau}(Z(\tau),Z'(\tau))-{\rm Ric}_{\tau}(Z(\tau),Z'(\tau))\nonumber\\
    &=0.
\end{align}

\begin{remark}\label{rem1}
The minimal $\mathcal{L}$-geodesic $\gamma=\gamma^{t_1,t_2}_{x,y}$ of $Q(x,t_1;y,t_2)$ satisfies the
$\mathcal{L}$-geodesic equation
$$\nabla^{t}_{\dot{\gamma}(t)}\dot{\gamma}(t)=\frac{1}{2}\nabla^{t}R_{t}-2{\rm Ric}^{\sharp}_{t}(\dot{\gamma}(t))-\frac{1}{2t}\dot{\gamma}(t).$$
Therefore, $\sqrt{t}\dot{\gamma}(t)$ is not space-time parallel to $\gamma$ in general and their $g_t$-inner product
is not a constant in $t$, i.e.  $\sqrt{t}\dot{\gamma}(t)$ does not satisfy \eqref{eq4}. Therefore the coefficient in the martingale part of  \eqref{martingale}
is not constant, which is different from the  case when the metric is independent of $t$.

\end{remark}
\begin{definition}[space-time parallel transport]\label{def2}
 For $x,y\in M$ and $0<\tau_1<\tau_2\leq T$, we define a map
$P_{x,y}^{\tau_1,\tau_2}: T_xM\rightarrow T_yM$ as follows: $P_{x,y}^{\tau_1\tau_2}(\xi):=Z(\tau_2)$,
where $Z$ is the unique space-time parallel vector field along $\gamma_{x,y}^{\tau_1\tau_2}$ with $Z(\tau_1)=\xi$.
As explained in \eqref{isometry}, $P_{x,y}^{\tau_1,\tau_2}$ is an isometry from $(T_xM, g_{\tau_1})$ to $(T_yM,g_{\tau_2})$.
In addition, it smoothly depends on $(x,\tau_1,y,\tau_2)$ outside the $\mathcal{L}$-cut locus.
\end{definition}
Using the It\^{o} formula for $Q(X_t,t)$ presented in Theorem \ref{Ito-formula},  we are able to construct a parallel
coupling of $g_{\bar{\tau}_1t}$- and $g_{\bar{\tau}_2t}$-Brownian motions.

\bthm \label{thm2}
 Let $x\neq y$ and
$0<\bar{\tau}_1<\bar{\tau}_2<T$ be fixed. For any $s>0$,
\ there exist two Brownian motions $B_t$ and
  $\tilde{B}_t$ on a completed filtered probability space $(\Omega,\{\mathscr{F}_t\}_{t\geq 0},\mathbb{P})$ such that for all $t\in [s,T/\bar{\tau}_2)$,
  $${\bf 1}_{\{({X}_t,\bar{\tau}_1t;\tilde{X}_t,\bar{\tau}_2t)\notin \mL{\rm
  Cut}\}}\vd \tilde{B}_t={\bf 1}_{\{({X}_t,\bar{\tau}_1t; \tilde{X}_t,\bar{\tau}_2t)\notin \mL{\rm
  Cut}\}}(\tilde{u}_t)^{-1}P_{{X}_t,\tilde{X}_t}^{\bar{\tau}_1t,\bar{\tau}_2t}u_{t}\vd B_t$$
holds, where ${X}_t$ with lift $u_{t}$ and
$\tilde{X}_{t}$ with lift $\tilde{u}_t$ solve the equation
\be\label{parra}\begin{cases} \vd
{X}_t=\sqrt{2\bar{\tau}_1}u_{t}\circ \vd B_t,&
X_{s}=x;\\
\vd \tilde{X}_t=\sqrt{2\bar{\tau}_2}\tilde{u}_t\circ \vd \tilde{B}_t,&
\tilde{X}_{s}=y.
\end{cases} \de
Moreover,
\begin{align}\label{eq5}
\vd
Q({X}_t,\bar{\tau}_1t;\tilde{X}_t, \bar{\tau}_2t)\leq & 2\sqrt{2t}\l|\bar{\tau}_1P_{X_t,\tilde{X}_t}^{\bar{\tau}_1t,\bar{\tau}_2t}\dot{\gamma}(\bar{\tau}_1t)-\bar{\tau}_2\dot{\gamma}(\bar{\tau}_2t)\r|_{\bar{\tau}_2t}\vd b_t\nonumber\\
&+\l(\frac{d}{\sqrt{t}}(\sqrt{\bar{\tau}_1}-\sqrt{\bar{\tau}_2})-\frac{1}{2t}Q(\bar{\tau}_1t,{X}_t;\bar{\tau}_2t,\tilde{X}_t)\r)\vd t,
\end{align}
where  $\gamma:[\bar{\tau}_1t,\bar{\tau}_2t]\rightarrow M$ is the $\mathcal{L}$-geodesic from
$X_t$ to $\tilde{X}_t$.

\ethm
\begin{proof}
We denote $Q(t,x,y):=Q(x, \bar{\tau}_1 t; y,\bar{\tau}_2t)$ for simplicity.
 Our proof is divided into two parts.

${ \mathbf{(a)}}$ \ First, we give the construction of the couplings.
Recall that $u_t$, the horizontal lift of $X_t$, satisfies the
following SDE
$$\begin{cases}
 \vd u_t=\sqrt{2\bar{\tau}_1}\dsum_{i=1}^{d}H_{i}(\bar{\tau}_1t,u_t)\circ \vd
B_t^{i}-\frac{1}{2}\bar{\tau}_1\dsum_{\alpha,
\beta}\mathcal{G}_{\alpha, \beta}(\bar{\tau}_1t, u_t)V_{\alpha \beta}(u_t)\vd t,\medskip\\
u_s\in \mathcal{O}_{\bar{\tau}_1s}(M),\ {\bf p}u_s=x.
\end{cases}$$
Now, for given $x\neq y$ with $(x, \bar{\tau}_1t;y, \bar{\tau}_2t)\notin \LC$, let
$\gamma$ be the minimal geodesic from $x$ to $y$. Recall that
$P^{\bar{\tau}_1t,\bar{\tau}_2t}_{x,y}$ are the parallel   operators.
To get rid of the trouble that $P^{\bar{\tau}_1t,\bar{\tau}_2t}_{x,y}$ does not exist on
$\LC$, we modify this operator
so that it vanishes in a neighborhood of this set. To this end,
for any $n\geq 1$ and $\varepsilon\in (0,1)$, let
$h_{n,\varepsilon}\in C^{\infty}([s,T/\bar{\tau}_2)\times M\times M)$ such that
$$0 \leq h_{n,\varepsilon}\leq (1-\varepsilon),\ h_{n,\varepsilon}|_{C_{2n}}=0,\ h_{n,\varepsilon}|_{C_n^c}={1-\varepsilon},$$
where ${C}_n=\{(t,x,y)\in [s,T/\bar{\tau}_2)\times M\times M: {\bf d}_{\vec{g}}\otimes {\bf d}_{\vec{g}}((x, \bar{\tau}_1t;y,\bar{\tau}_2t),\LC)\leq 1/n\}$.
Let $\tilde{u}^{n,\varepsilon}_t$ and
$\tilde{X}_t^{n,\varepsilon}:={\bf
p}\tilde{u}^{n,\varepsilon}_t$ solve the SDE
\begin{equation}\label{SDE2}
\begin{cases}
 \vd \tilde{u}_t^{n,\varepsilon}=\sqrt{2\bar{\tau}_2}h_{n,\varepsilon}(t,X_t,\tilde{X}_t^{n,\varepsilon})\dsum_{i=1}^{d}H_{i}(\bar{\tau}_2 t,\tilde{u}_t^{n,\varepsilon})\circ \vd
\tilde{B}_t^{i}-\frac{1}{2}\bar{\tau}_2\dsum_{\alpha,
\beta}\mathcal{G}_{\alpha, \beta}(\bar{\tau}_2t, \tilde{u}_t^{n,\varepsilon})V_{\alpha \beta}(\tilde{u}_t^{n,\varepsilon})\vd t\\
\qquad \qquad
+\sqrt{2\bar{\tau}_2[1-(h_{n,\varepsilon})^2(t,X_t,\tilde{X}_t^{n,\varepsilon})]}\dsum_{i=1}^{d}H_{i}(\bar{\tau}_2t,\tilde{u}_t^{n,\varepsilon})\circ
\vd {B'_t}^{i},\medskip\\
\tilde{u}_s^{n,\varepsilon}\in \mathcal{O}_{\bar{\tau}_2s}(M),\ {\bf
p}\tilde{u}_s^{n,\varepsilon}=y,
\end{cases}
\end{equation}
where $B'_t$ is a Brownian motion on $\mathbb{R}^d$ independent of
$B_t$, and $\vd
\tilde{B}_t={(\tilde{u}^{n,\varepsilon}_t)}^{-1}P^{\bar{\tau}_1t,\bar{\tau}_2t}_{X_t,
\tilde{X}_t^{n,\varepsilon}}u_t\vd B_t $. Since the coefficients
involved in (\ref{SDE2}) are at least $C^{1}$, the solution
$\tilde{u}^{n,\varepsilon}_t$ exists uniquely.

Let us observe that $(u_t,\tilde{u}^{n,\varepsilon}_t)$ is
generated by
\begin{align*}
&L^{n,\varepsilon}_{\mathcal{O}(M\times M)}(t)(u_t,\tilde{u}^{n,\varepsilon}_t)\\
&:=\bar{\tau}_1\Delta_{\mathcal{O}_{\bar{\tau}_1t}(M)}(u_t)+\bar{\tau}_2\Delta_{\mathcal{O}_{\bar{\tau}_2t}(M)}(\tilde{u}^{n,\varepsilon}_t)\\
&\quad +\sqrt{\bar{\tau}_1 \bar{\tau}_2}h_{n,\varepsilon}(t,X_t,\tilde{X}_t^{n,\varepsilon})\sum_{i,j=1}^{d}\l<P^{\bar{\tau}_1t,\bar{\tau}_2t}_{X_t,\tilde{X}_t^{n,\varepsilon}}u_te_i,
\tilde{u}^{n,\varepsilon}_te_j\r>_{\bar{\tau}_2t}H(\bar{\tau}_1t,u_t)H(\bar{\tau}_2t,\tilde{u}^{n,\varepsilon}_t)\\
&\quad-\frac{1}{2}\bar{\tau}_1\dsum_{\alpha,
\beta}\mathcal{G}_{\alpha, \beta}(\bar{\tau}_1t, u_t)V_{\alpha \beta}(u_t)\vd t-\frac{1}{2}\bar{\tau}_2\dsum_{\alpha,
\beta}\mathcal{G}_{\alpha, \beta}(\bar{\tau}_2t, \tilde{u}_t^{n,\varepsilon})V_{\alpha \beta}(\tilde{u}_t^{n,\varepsilon})\vd t.
\end{align*}
Next, let
\begin{align*}
L^{n,\varepsilon}_{M\times M}(t)(x,y):=&\bar{\tau}_1\Delta_{\bar{\tau}_1t}(x)+\bar{\tau}_2\Delta_{\bar{\tau}_2t}(y)
+\sqrt{\bar{\tau}_1 \bar{\tau}_2}h_{n,\varepsilon}(t,x,y)\sum_{i,j=1}^{d}\l<P^{\bar{\tau}_1t,\bar{\tau}_2t}_{x,y}V_i,
W_j\r>_{\bar{\tau}_2t}V_iW_j,
\end{align*}
where $\{V_i\}$ and $\{W_i\}$ are orthonormal bases at $x$ and $y$ with respect to the metrics $g_{\bar{\tau}_1t}$ and $g_{\bar{\tau}_2t}$ respectively.
 It is easy to see that
$(X_t,\tilde{X}_t^{n,\varepsilon}):=({\bf p}u_t,{\bf p} \tilde{u}_t^{n,\varepsilon})$ is generated by
$L_{M\times M}^{n,\varepsilon}(t)$ and hence is a coupling of $g_{\bar{\tau}_1t}$- and $g_{\bar{\tau}_2t}$-Brownian motions, as the marginal operators of
$L_{M\times M}^{n,\varepsilon}(t)$ coincide with $\bar{\tau}_1\Delta_{\bar{\tau}_1t}$ and $\bar{\tau}_2\Delta_{\bar{\tau}_2t}$ respectively.

Since in some neighborhood of $\LC$,
the coupling is independent and hence behaves as a Brownian
motion  on $M\times M$, we obtain from Theorem \ref{Ito-formula} that
\begin{align*}
\vd Q(t,X_t,\tilde{X}_t^{n,\varepsilon})=&
\sqrt{8t\l|\bar{\tau}_1P_{X_t,\tilde{X}_t^{n,\varepsilon}}^{\bar{\tau}_1t,\bar{\tau}_2t}\dot{\gamma}_{n,\varepsilon}(\bar{\tau}_1t)-\bar{\tau}_2h_{n,\varepsilon}\dot{\gamma}_{n,\varepsilon}(\bar{\tau}_2t)\r|_{\bar{\tau}_2t}^2+8t(1-h_{n,\varepsilon}^2)\bar{\tau}_2^2|\dot{\gamma}_{n,\varepsilon}(\bar{\tau}_2t)|^2_{\bar{\tau}_2t}}
\ \vd b^{n,\varepsilon}_t\nonumber\\
&+\l\{\frac{\partial Q}{\partial t}+{\bf 1}_{(\Upsilon\setminus\LC)}[h_{n,\varepsilon}I+(1-h_{n,\varepsilon})S]\r\}(t,
X_t,\tilde{X}_t^{n,\varepsilon})\vd t-\vd l_t^{n,\varepsilon},
\end{align*}
where  $\gamma_{n,\varepsilon}:[\bar{\tau}_1t,\bar{\tau}_2t]\rightarrow M$ is the $\mathcal{L}$-geodesic from
$X_t$ to $\tilde{X}^{n,\varepsilon}_t$,  $b_t^{n,\varepsilon}$ is an one-dimensional Brownian motion,
$l_t^{n,\varepsilon}$ is an increasing process which increases only
when $(X_t, \bar{\tau}_1t;\tilde{X}_t^{n,\varepsilon},\bar{\tau}_2t)\in \LC$,
 and
\begin{align*}
&S(t,x,y):=\bar{\tau}_1\Delta_{\bar{\tau}_1t}Q(t,\cdot,
y)(x)+\bar{\tau}_2\Delta_{\bar{\tau}_2t}Q(t,x,\cdot)(y);\\
& I(t,x,y):=\sum_{i=1}^{d}(\sqrt{\bar{\tau}_1}V_i+\sqrt{\bar{\tau}_2}P^{\bar{\tau}_1t,\bar{\tau}_2t}_{x,y}V_i)^2Q(t,x,y).
\end{align*}
Then, let $\mathbb{P}^{(x,y)}_{n,\varepsilon}$ be the distribution
of $(X_t,\tilde{X}_t^{n,\varepsilon})$, which is a probability
measure on the path space $M^{T}_x\times M^{T}_y$, where
$$M_x^{T}:=\{\gamma\in C([s,T/\bar{\tau}_2), M):\gamma_s=x\}$$
is equipped with the $\sigma$-field $\mathscr{F}_{x}^{T}$
induced by all measurable cylindric functions. Note that
$(M_x^{T},\mathscr{F}_x^{T})$ is metrizable with the
distance
$$\tilde{\rho}(\xi,\eta):=\sum_{n=1}^{\infty}2^{-n}\l(1\wedge\sup_{t\in [n,(n+t)\wedge T/\bar{\tau}_2)}{\bf d}_{\vec{g}}(\bar{\tau}_1t,\xi_t; \bar{\tau}_2t,\eta_t)\r).$$
Furthermore, $(M_x^{T},\tilde{\rho})$ is a Polish space.
Then $M_x^{T}\times M_y^{T}$ is a Polish space too. It is
easy to see that $\{\mathbb{P}^{x,y}_{n,\varepsilon}: n\geq
1,\varepsilon>0\}$ is tight (see \cite[Lemma 4]{Re}), since they are the couplings of
$\mathbb{P}^x$ and $\mathbb{P}^y$. We take $n_k\rightarrow\infty$ and $\varepsilon_l\rightarrow 0$ such
 that $\mathbb{P}^{x,y}_{n_k,\varepsilon_l}$ converges weakly to some
 $\mathbb{P}_{\varepsilon_l}^{x,y}$ ($l\geq 1$) as $k\rightarrow\infty$ while $\mathbb{P}^{x,y}_{\varepsilon_l}$
 converges weakly to some $\mathbb{P}^{x,y}$ as $l\rightarrow\infty$. Then $\mathbb{P}^{x,y}$ is also a coupling of $\mathbb{P}^x$
and $\mathbb{P}^y$.

Now, let $({X}_t, {\tilde{X}}_t)$ be the coordinate  process in $(M_x^{T}\times M_y^{T},
\mathscr{F}_x^{T}\times \mathscr{F}_y^{T})$ and
$\{\mathscr{F}_t\}_{t\geq s}$ be the natural filtration.   Define
$$\tilde{L}(t)(x,y):=\bar{\tau}_1\Delta_{\bar{\tau}_1t}(x)+\bar{\tau}_2\Delta_{\bar{\tau}_2t}(y)+1_{(\Upsilon\setminus\LC)}(x,\bar{\tau}_1t;y, \bar{\tau}_2t)\sum_{i,j=1}^d\l<P_{x,y}^{\bar{\tau}_1t,\bar{\tau}_2t}V_i,W_j\r>_{\bar{\tau}_2t}V_iW_j,$$
It is trivial to see that $\mathbb{P}^{x,y}$ solves the martingale problem for $\tilde{L}(t)$
 up to $T/\bar{\tau}_2$ (see \cite[Theorem 2]{ATW}), i.e.
 $$f({X}_t,\tilde{X}_t)-\int_0^t\tilde{L}(s)f(X_s,\tilde{X}_s)\vd s$$
 is a $\mathbb{P}^{x,y}$-martingale w.r.t.  $\mathscr{F}$.
 Then $({X}_t,\tilde{X}_t)$ under $\mathbb{P}^{x,y}$ is a coupling of the
$g_{\bar{\tau}_1t}$- and $g_{\bar{\tau}_2t}$-Brownian motions starting from $(x,y)$, i.e. the solution of
(\ref{parra}).
\medskip

 $\mathbf{(b)}$\
  We first claim that the two sets
\begin{align*}
\{t\in[s,T/\bar{\tau}_2)\ |\ (X_t,\bar{\tau}_1t;\tilde{X}_t^{n,\varepsilon}, \bar{\tau}_2t)\in \LC\}\ \mbox{and}\
\{t\in[s,T/\bar{\tau}_2)\ | \ (X_t, \bar{\tau}_1t;\tilde{X}_t^{\varepsilon}, \bar{\tau}_2t)\in \LC\}
\end{align*} have Lebesgue measure zero almost surely.
This assertion can be checked similarly as in Lemma \ref{pmea0}
by observing that $L^{n,\varepsilon}_{M\times M}(t)$ is strictly elliptic and all the coefficients are $C^{\infty}$,
then $\mathbb{P}^{x,y}((X_t,
\tilde{X}_t^{n,\varepsilon})\in A)$ has a density $p^{n,
\varepsilon}_t(x,y;z,w)$ with respect to the product volume measure  ${\rm d} {\rm vol}_{\bar{\tau}_1t} \otimes {\rm d} {\rm vol}_{\bar{\tau}_2t}$ (see \cite[Theorem 3.16]{Fr}),  where ${\rm d}{\rm vol}_{t}$ is the volume measure w.r.t. the metric $g_{t}$.
Then, for $f\in C^2(\mathbb{R})$ with $f'\geq 0$ be fixed, let
\begin{align*}
\vd N_t(f)=& \vd f\circ Q(t,X_t,\tilde{X}_t)-\bigg[\l(I+\frac{\partial Q}{\partial t}\r)f'\circ Q(t,X_t,\tilde{X}_t)\\
&+4t\l|\bar{\tau}_1P_{X_t,\tilde{X}_t}^{\bar{\tau}_1t,\bar{\tau}_2t}\dot{\gamma}(\bar{\tau}_1t)-\bar{\tau}_2\dot{\gamma}(\bar{\tau}_2t)\r|^2_{\bar{\tau}_2t}f''\circ Q (t,X_t,\tilde{X}_t)\bigg]\vd t,\ \ t\in [s,T/\bar{\tau}_2).
\end{align*}
With a similar discussion as in the proof of \cite{cheng}$(b)$, we obtain
$N_{t\wedge (T/\bar{\tau}_2)}$ is a $\mathbb{P}^{x,y}$-supmartingale.
In particular, by taking explicit $f(r)=r$,
we have
\begin{align}\label{Ito-Q}
\vd Q(t, X_t,\tilde{X}_t)=\vd M_t +\l(I+\frac{\partial Q}{\partial t}\r)(t,X_t,\tilde{X}_t)\vd t-\vd l_t,\ \ t\in [s,T/\bar{\tau}_2),
 \end{align}
 where $M_t$ is a local martingale and $l_t$ is a predictable
increasing process.
By the second variation  formula of the   $\mathcal{L}$-functional (see \cite[Lemma 7.37 and Lemma 7.40]{Ricci} for instance),
$$\frac{\partial}{\partial t}Q(t,x,y)+\sum_{i=1}^d(\sqrt{\bar{\tau}_1}V_i+\sqrt{\bar{\tau}_2}P_{x,y}^{\bar{\tau}_1t,\bar{\tau}_2t}V_i)^2Q(t,x,y)\leq \frac{d}{\sqrt{t}}(\sqrt{\bar{\tau}_1}-\sqrt{\bar{\tau}_2})-\frac{1}{2t}Q(t,x,y)=: J(t,x,y).$$
It follows that
 \begin{align}\label{Ito-Q}
\vd Q(t, X_t,\tilde{X}_t)=\vd M_t +J(t,X_t,\tilde{X}_t)\vd t-\vd \tilde{l}_t,\ \ t\in [s,T/\bar{\tau}_2),
 \end{align}
where   $\tilde{l}_t$ is a larger  predictable
increasing process.  Moreover,
with a similar discussion as in the proof of  \cite[Theorem 2.1.1]{Wbook1}(d), we further obtain
$$\vd M_t=2\sqrt{2t}\l|\bar{\tau}_1P_{X_t,\tilde{X}_t}^{\bar{\tau}_1t,\bar{\tau}_2t}\dot{\gamma}(\bar{\tau}_1t)-\bar{\tau}_2\dot{\gamma}(\bar{\tau}_2t)\r|_{\bar{\tau}_2t}\vd b_t,$$
which leads to complete the proof.


\end{proof}

As an important application, we give a new proof of Topping's result \cite{Topping}, i.e. the contraction in the normalized
$\mathcal{L}$-transportation cost. We point out that this result recovers the monotonicity
of Perelman's monotonic quantities (involving both $\mathcal{W}$-entropy and $\mathcal{L}$-length), which are central in his work on Ricci flow (see \cite{Pe1,Pe2,Pe3}).

Suppose that $\{P_{s,t}\}_{0<s<t<T/\bar{\tau}_2}$ and $\{T_{s,t}\}_{0<s<t<T/\bar{\tau}_2}$ be the Markov inhomogeneous semigroup of the $g_{\bar{\tau}_1t}$-Brownian motion and $g_{\bar{\tau}_2t}$-Brownian motion respectively.
To the $\mathcal{L}$-distance $Q$,
we associate the Monge-Kantorovich minimization between
two probability measures on $M$,
  \begin{align*}
  W^{\mathcal{L}}(\mu,t_1;\nu,t_2)=\inf_{\eta\in \mathscr{C}(\mu,\nu)}\int_{M\times M}Q(x,t_1;y,t_2)\vd \eta(x,y),\end{align*}
  where $\mathscr{C}(\mu,\nu)$ is the set of all probability measures on $M\times M$ with marginal $\mu$ and
$\nu$.
Then, using the coupling constructed
in Theorem \ref{thm2}, we have
\begin{theorem}\label{L-optimal-transportation}
Assume that $M$ has bounded curvature tensor, i.e.
$$\sup_{x\in M, t\in [0,T)}|{\rm Rm}_{t}|(x)<\infty.$$
Then for $0<\bar{\tau}_1<\bar{\tau}_2<T$ the normalized $\mathcal{L}$-transportation cost
$$\Theta(t, \delta_xP_{s,t},\delta_yT_{s,t}):=2(\sqrt{\bar{\tau}_2t}-\sqrt{\bar{\tau}_1t})W^{\mathcal{L}}(\delta_xP_{s,t},\bar{\tau}_1t;\delta_yT_{s,t},\bar{\tau}_2t)-2d(\sqrt{\bar{\tau}_2t}-\sqrt{\bar{\tau}_1t})^2$$
is a non-increasing function of $t\in [s,T/\bar{\tau}_2)$, that is
$$\Theta(t,\delta_xP_{s,t},\delta_yT_{s,t})\leq 2(\sqrt{\bar{\tau}_2s}-\sqrt{\bar{\tau}_1s})Q(x,\bar{\tau}_1s;y,\bar{\tau}_2s)-2d(\sqrt{\bar{\tau}_2s }-\sqrt{\bar{\tau}_1s})^2.$$
\end{theorem}
\begin{proof}
By \eqref{parra}, there exist two coupled $g_{\bar{\tau}_1t}$- and $g_{\bar{\tau}_2t}$-Brownian motions $(X_t)_{t\in [s,T/\bar{\tau}_2)}$
and $(\tilde{X}_t)_{t\in [s,T/\bar{\tau}_2)}$ with initial values $X_{s}=x$ and $\tilde{X}_{s}=y$ such that
the process $(\Theta(t,X_{\bar{\tau}_1t},\tilde{X}_{\bar{\tau}_2t}))_{t\in [s,T/\bar{\tau}_2)}$ is a supermartingale. Taking the expectation of this supermartingale leads to complete the proof.
\end{proof}
   Denote   $P_{s,t}(\cdot,\vd x):=u(\bar{\tau}_1t,x)\vd {\rm vol}_{\bar{\tau}_1t}$ and $T_{s,t}(\cdot,\vd x):=u(\bar{\tau}_2t,x)\vd {\rm vol}_{\bar{\tau}_2t}$. This density $u$ solves the following heat equation
$$\frac{\partial u}{\partial \tau}=\Delta_{\tau}u-R_{\tau}u,$$
where $R_{\tau}$ is the scalar curvature w.r.t. the metric $g_{\tau}$.
Hence, the conclusion presented in Theorem \ref{L-optimal-transportation} is consistent with  that in \cite{Topping}.
\bigskip
\bigskip

\noindent\textbf{Acknowledgements}  \ The author  thank Professor Feng-Yu Wang for valuable suggestions and  this work is supported in part
by NNSFC(11131003), SRFDP, and the Fundamental Research Funds for the
Central Universities.

\end{document}